\documentclass[12pt,a4paper]{article}
\usepackage[OT1]{fontenc}
\usepackage[english]{babel}
\usepackage{amsfonts}
\usepackage{amsthm}
\usepackage{amsmath}

\usepackage{authoraftertitle}
\usepackage{fancyhdr}
\usepackage{color}

\def\h{ {\cal H} }
\def\l{ {\cal L} }
\def\a{ {\cal A} }

\def\b{ {\cal B} }
\def\u{ {\cal U} }
\def\v{ {\cal V} }

\def\ii{ {\cal I} }
\def\t{ {\cal T} }
\def\s{ {\cal S} }

\def\p{ {\cal P} }

\def\k{ {\cal K} }

\def\d{ {\cal D} }

\def\o{ {\cal O} }

\def\XX{ {\bf X} }
\def\TT{ {\bf T} }
\def\SS{ {\bf S} }

\def\BB{ {\bf B} }
\def\AA{ {\bf A} }
\def\GG{ {\bf G} }

\def\PP{ {\bf P} }
\def\QQ{ {\bf Q} }
\def\ZZ{ {\bf Z} }
\def\VV{ {\bf V} }

\def\d{\displaystyle}
\def\HH{{\mathcal{H}}}

\newtheorem{teo}{Theorem}[section]
\newtheorem{prop}[teo]{Proposition}
\newtheorem{lem}[teo]{Lemma}

\theoremstyle{definition}
\newtheorem{rem}[teo]{Remark}
\newtheorem{ejem}[teo]{Example}
\newtheorem{ejems}[teo]{Examples}

\title{Operators which preserve a positive definite inner product}
\author{E. Andruchow\footnote{{\sc  {Instituto Argentino de Matem\'atica, `Alberto P. Calder\'on', CONICET, Saavedra 15 3er. piso,
(1083) Buenos Aires, Argentina }} and {{\sc Universidad Nacional de General Sarmiento, J.M. Gutierrez 1150, (1613) Los Polvorines, Argentina}} e-mail: eanddruch@campus.ungs.edu.ar}
}

\begin{document}

\maketitle 

\begin{abstract}
Let $\h$ be a Hilbert space, $A$  a positive definite operator in $\h$ and  $\langle f,g\rangle_A=\langle Af,g\rangle$, $f,g\in\h$, the $A$-inner product. This paper studies the geometry of the set
$$
\ii_A^a:=\{\hbox{ adjointable isometries for } \langle \ , \ \rangle_A\}.
$$
It is proved that $\ii_A^a$ is a submanifold of the Banach algebra of adjointable operators, and a homogeneous space of the group of invertible operators in $\h$, which are unitaries for the $A$-inner product. Smooth curves in $\ii_A^a$ with given initial conditions, which are minimal for the metric induced by $\langle \ , \ \rangle_A$, are presented. This result depends on an adaptation of M.G. Krein's extension method of symmetric contractions, in order that it works also for symmetrizable transformations (i.e., operators which are selfadjoint for the $A$-inner product).
\end{abstract}

\bigskip



\section{Introduction}
Let $A$ be a positive contraction in the Hilbert space $(\h,\langle \  \ , \ \ \rangle)$ with trivial nullspace $N(A)=\{0\}$. Denote by $\langle\ \ , \ \ \rangle_A$ the inner product defined by $A$:
$$
\langle f,g\rangle_A=\langle Af,g\rangle, \ f,g\in\h.
$$
This paper considers the operators which preserve this inner product, which will be called  $A$-isometries. 
That is,  a bounded operator $T\in\b(\h)$ is called an $A$-isometry if 
$$
\langle ATf,Tg\rangle=\langle A f,g\rangle,
$$
or equivalently,  $T^*AT=A$. The focus will be on $A$-isometries which additionally admit an adjoint for the $A$-inner product: 
 there exists $S\in\b(\h)$  such that $\langle Tf,g\rangle_A=\langle f, Sg\rangle_A$, for $ f,g\in\h$. 
 In general, an operator $B$ in $\h$ will be called $A$-{\it adjointable} (or just {\it adjointable}) if it admits an adjoint for the $A$-inner product, which  shall be  denoted by $C=A^\sharp$. Denote by
$$
\ii_A:=\{T\in\b(\h): T \hbox{ is } A-\hbox{isometric}\}
$$
and
\begin{equation}\label{conjunto}
\ii_A^a:=\{T\in\ii_A: T \hbox{ is } A-\hbox{adjointable}\}
\end{equation}

It will be shown that there are $A$-isometries which are not $A$-adjointable. It is the latter set $\ii_A^a$ which admits a differentiable structure, and which is a homogeneous space of a Banach-Lie group. Moreover, endowed with a natural Finsler metric (namely, the one induced by he $A$-inner product),  curves which have minimal length for given initial conditions are computed.

The contents of the paper are the following. In Section 2  the basic facts needed are stated, on operators in Hilbert spaces with two norms (these facts are taken from \cite{krein}, \cite{lax}, \cite{gohberg zambickii});   examples of adjointable and non adjointable $A$-isometries are presented. Also Theorem \ref{proposicion 22} is proved, characterizing adjointable $A$-isometries. A result by R. Douglas \cite{douglas} is used, on existence of solutions of the operator equation $AX=B$. In Section 3  the Wold decomposition of $A$-isometries is briefly analized. Section 4 contains Theorem \ref{estructura diferenciable}, stating  that $\ii_A^a$ is a $C^\infty$-submanifold of the $*$-Banach algebra of adjointable operators (with a suitable norm), and a homogeneous space of the Banach-Lie group of $A$-unitary operators (see the definition below). Section 5 presents an adaptation of   Krein's method for the extension of symmetric transformations with norm constraints, to work in the context of symmetrizable transformations (Lemma \ref{krein method}). We believe that this section is of independent interest. This result is used in Section 6 to prove Theorem \ref{minimalidad}, which computes minimal curves in $\ii_A^a$ satisfying given initial conditions. Section 7 treats the action of the restricted group of $A$-unitaries,  the orbits of the action of this group are characterized (Theorem \ref{orbitas restringidas}), and it is proved that these are also $C^\infty$-manifolds and homogeneous spaces (Proposition \ref{variedad restringida}).

Let us finish this section with basic facts and notations.
Denote by $\l$ the completion of the pre-Hilbert space $(\h,\langle \ \ , \ \ \rangle_A)$. Note that $T\in\ii_A$ extends to an isometry $\TT$ acting in $\l$ (denote the inner product on $\l$ by $\langle\ \ , \ \ \rangle_\l$). Conversely, an isometry $\VV\in\b(\l)$  that   leaves the dense subspace $\h$ invariant, i.e. $\VV(\h)\subset\h$,  induces an operator $V=\VV|_\h$ in $\h$, which is an $A$-isometry.

This study relies on the theory of symmetrizable operators, (or more broadly, operators in Hilbert spaces with two norms), developed initially (and independently) by M.G. Krein \cite{krein} and P.D. Lax \cite{lax}, and extended afterwards by I.C.  Gohberg and M.K. Zambickii \cite{gohberg zambickii}.

Upper case letters $T, S, X, G, \dots$ will denote operators acting in $\h$ (with adjoints in $\h$ denoted $T^*,S^*, X^*, G^*,\dots$). Their eventual extensions to $\l$  will be denoted with capital {\bf bold} letters: $\TT, \SS, \XX, \GG,\dots$, and adjoints $\TT^*, \SS^*, \XX^*, \GG^*,\dots$ (in $\l$). That is, which adjoint is refered to, will depend on the context. Vectors in $\h$ will be denoted $f,g,h$ and  vectors in $\l$ with Greek letters $\varphi, \gamma, \eta$.

A special class of $A$-isometries, is given by the $A$-unitary operators: $G\in\b(\h)$ is called $A$-unitary if it is an $A$-isometry which is invertible in $\h$. Denote by 
$$
\u_A=\{ G\in\b(\h): G \hbox{ is invertible and } G^*AG=A\},
$$
the group of $A$-unitaries. Note that $A$-unitaries are adjointable. In \cite{grupo dos normas} it was shown that $\u_A$ is a $C^\infty$ Banach-Lie group. Clearly, $\u_\a$ acts on $\ii_A$ and on $\ii_A^a$ by left multiplication:
$$
G\cdot T=GT\in\ii_A \ , \  \hbox{ if }   G\in\u_A \hbox{ and } T\in\ii_A,
$$ 
and clearly $GT$ is adjointable if $T$ is adjointable.
Again, $G\in\u_A$ extends to a unitary operator $\GG$ such that $\GG(\h)=\h$.

\section{Symmetrizable operators}

Denote by $\sigma_\h(T)$ the spectrum of $T$ as an operator in $\h$, and by $\sigma_\l(\TT)$ the spectrum of its extension to $\l$ (when it exists). We will say that  $\lambda$  belongs to the point spectrum of $T$  if $0<\dim(N(T-\lambda 1))$, and $\lambda$ is said to have finite multiplicity if this dimension is finite. An operator $B$ acting in $\h$ will be called $A$-symmetric (or symmetrizable) if it is symmetric for the $A$-inner product. Let us recall the following facts, adapted from their original broader context to our case:
\begin{teo} {\rm (See M.G.Krein \cite{kreindosnormas}, P.D. Lax \cite{lax})} Let $B$ be an $A$-symmetric operator. The following assertions hold:
\begin{enumerate}
\item
$\BB$ exists and is bounded selfadjoint operator in $\l$.
\item
$\sigma_\l(\BB)\subset\sigma_\h(B)$.
\item
If $\lambda$ belongs to the point spectrum of $B$ as an operator in $\h$, then $\lambda$ belongs to the point spectrum of $\BB$ as an operator in $\l$. Moreover, if $\lambda$ has finite multiplicity, then the $\lambda$-eigenspaces over $\h$ and $\l$ are the same.
\item
If $B$ is a compact operator in $\h$, then $\BB$ is a compact operator in $\l$ and $\sigma_\l(\BB)=\sigma_\h(B)$.
\end{enumerate}
\end{teo}
\begin{rem}\label{gohberg}
I.C.  Gohberg and M.I. Zambickii \cite{gohberg zambickii} extended these results to operators in Banach spaces with two norms. In particular, in our context, they proved that if an operator $B$ in $\h$ is compact and $A$-adjointable (they called adjointable operators {\it proper}), the its extension $\BB$ to $\l$ is also compact, with the same point spectrum, and also the same eigenspaces of finite multiplicity.
\end{rem}
Denote by
\begin{equation}\label{adjuntables}
\b_A(\h)=\{B\in\b(\h): B \hbox{ is } A-\hbox{adjointable}\}
\end{equation}
If $A$ is not invertible, $\b_A(\h)$ is not closed in $\b(\h)$, neither is the $A$-adjoint map
$B\mapsto B^\sharp$  a continuous map (in the norm topology).  $\b_A(\h)$ is endowed with he norm
$$
|B|:=\max\{\|B\|,\|B^\sharp\|\}.
$$
Then $(\b_A(\h), |\ \ |)$ becomes an involutive Banach algebra.
\begin{rem}
It is not difficult to see that if $B$ is $A$-symmetric, then $\|\BB\|\le\|B\|$. Also note that $A$ itself is $A$-symmetric, and that its extension $\AA$ remains positive definite.
\end{rem}

A closed linear subspace $\s\subset\h$ is called {\it compatible with} $A$,  $A$-{\it compatible}, or shortly {\it compatible},  if it admits a supplement which is orthogonal with respect to the inner product defined by $A$. In \cite{compatible grassmannian}, the compatible Grassmannian was studied, namely
$$
Gr_A = \{\s\subset\h:  \s \hbox{  is closed and compatible with } A\}.
$$
Since $A$ has trivial nullspace, if $S$ is compatible with $A$, then the supplement is unique, and it is given by $A(\s)^\perp$. This allows us to identify each compatible subspace $\s$ with the idempotent $Q_\s$  with range $\s$  and nullspace $A(\s)^\perp$. Thus the compatible Grassmannian may be regarded as the following set
$$
Gr_A =\{Q \in \b(\h): Q^2 = Q, Q^*A=AQ\}
$$
$$
=\{Q\in\b(\h): Q^2=Q \hbox{ is adjointable and symmetrizable}\}.
$$
In other words, $\QQ=P_{\overline{R(Q)}}$.
Consider in $Gr_A$ the  topology inherited from the norm of $\b_A(\h)$. The proof of the afore-mentioned facts and examples of compatible and non-compatible subspaces can be found in \cite{cms1}, \cite{cms2}, where a systematic study of compatible subspaces was done. The notion of compatible subspaces goes back to A. Sard \cite{sard}, who introduced an equivalent definition under a different terminology, to give an operator theoretic approach to problems in approximation theory (see \cite{cgm}). In \cite{compatible grassmannian} it was shown that $Gr_A$ is a complemented submanifold of $\b_A(\h)$, and an homogeneous space of the group $\u_A$ under the action 

%
$$G\cdot \s=G(\s) $$
or equivalently
$$G\cdot Q_\s=GQ_\s G^{-1},
\  \ G\in\u_A,
\ \s\in Gr_A.$$

An $A$-isometry may not be an adjointable operator. \begin{ejems}

\noindent

\begin{enumerate}
\item
Let $\l=H^2(\mathbb{T})$ the Hardy space of the unit circle $\mathbb{T}$, $\h=\d$ the Dirichlet space of the unit disk $\mathbb{D}$, i.e.,
$$\l=\{f:\mathbb{D}\to \mathbb{C}: f \hbox{ is analytic and } \sum_{n=0}^\infty |\hat{f}(n)|^2<\infty\} ,$$
$$\h=\{f\in\l: \sum_{n=0}^\infty (n+1) |\hat{f}(n)|^2<\infty\}.$$
Here the inner product of $\h$ is 
$$
\langle f,g\rangle=\sum_{n=0}^\infty (n+1)\hat{f}(n)\overline{\hat{g}(n)},
$$ 
and the inner product of $\l$ is 
$$
\langle f,g\rangle=\sum_{n=0}^\infty \hat{f}(n)\overline{\hat{g}(n)}, 
$$
in both cases for $f,g\in\h$. So that $A\in\b(\h)$ is given by
$$
Af(z)=\sum_{n=0}^\infty \frac{1}{n+1}\hat{f}(n)z^n.
$$
Note that $A$ is compact.
Let $\TT=M_z\in\b(\l)$ be the shift operator. Clearly $\TT(\h)\subset \h$ and $\TT^*(\h)\subset\h$, so that $T=\TT|_\h$ is an adjointable $A$-isometry. In particular,  see below (Theorem \ref{proposicion 22}),  this means that $R(T)=z \h$ is a compatible subspace.
\item
Consider the slight modification of  the above example. Put $\l=\ell^2$ and $\h=\{(x_n)\in\ell^2: \sum_{n=1}^\infty n|x_n|^2<\infty\}$, with its natural inner product $\langle(x_n),(y_n)\rangle_{_\h}=\sum_{n=1}^\infty n x_n \bar{y}_n$. Let $e_n$, $n\ge 1$ be the canonical orthonormal basis of $\ell^2$. Let $\mathbb{A}$ the subset of the natural numbers which are squares of odd integers,
$$
\mathbb{A}:=\{(2k+1)^2: k\ge 0\},
$$

and denote its complement 
$$
\mathbb{A}^c=\{\sigma(1)<\sigma(2)<\sigma(3)<\dots\}.
$$
Let $U^*$ be the unitary operator in $\ell^2$ given by
$$
U^*e_n=\left\{ \begin{array}{cc} e_{n^2} & \hbox{ if } n \hbox{ is odd,} \\
e_{\sigma(n/2)} & \hbox{ if } n \hbox{ is even.} \end{array} \right.
$$
Clearly $\{U^*e_n: n\ge 1\}=\{e_n: n\ge 1\}$, so that $U^*$ is indeed a unitary operator in $\ell^2$. Its adjoint $U$ is given by (for $x=(x_n)\in\ell^2$)
$$
(Ux)_n=\left\{ \begin{array}{cc} x_{n^2} & \hbox{ if } n \hbox{ is odd}, \\
x_{\sigma(n/2)} & \hbox{ if } n \hbox{ is even.} \end{array} \right.
$$
We claim that $U(\h)\subset\h$, but $U^*(\h)$ is not contained in $\h$. Which means that $U|_\h\in \ii_A$, but it is not $A$-adjointable.
For the first assertion, note the fact that for all $k\ge 1$, $k\le\sigma(k)\le 2k$.
Then, if $(x_n)\in\h$,
$$\begin{array}{cclcl}
\langle Ux,Ux\rangle_{_\h} &=&
\sum\limits_{n \ \textnormal{\scriptsize{odd}}} n|x_{n^2}|^2 &+&
\sum\limits_{n \ \textnormal{\scriptsize{even}}} n |x_{\sigma(n/2)} |^2 \\
&\leq& \sum\limits_{n \ \textnormal{\scriptsize{odd}}} n^2|x_{n^2}|^2
&+&
\sum\limits_{n \ \textnormal{\scriptsize{even}}} 2\sigma(n/2) |x_{\sigma(n/2)}|^2\\
&\leq& 2 \langle x,x\rangle_{_\h}. & 
\end{array}$$
To prove that $\h$ is not invariant for $U^*$, pick
$x_k = \frac{1}{k^{3/2}}$
if $k$ is odd
and $x_k=0$ otherwise.
%
Then clearly, $x = (x_n) \in \HH$,
and
$$(U^*x)_n=\left\{
 \begin{array}{cc}
\frac{1}{n^{3/4}} & \hbox{ if } n \in \mathbb{A}, \\
0 & \hbox{ if } n \in \mathbb{A}^c
 \end{array} \right. $$
Or equivalently,
$(U^*x)_{(2k+1)^2} = \frac{1}{(2k+1)^{3/2}}$
and equal to $0$ in all other entries.
Then
$$\langle U^*x, U^*x\rangle_{\HH} = 
\sum_{k=0} (2k+1)^2 \left( \frac{1}{(2k+1)^{3/2}} \right)^2 =
\sum_{k=0}\frac{1}{(2k+1)}=+\infty.$$

\end{enumerate}
\end{ejems}

Consider the subclass of $A$-adjointable $A$-isometries,
or shortly, {\it adjointable isometries}. Operators $G\in\u_A$ are examples of adjointable isometries. It is proved below  that adjointable isometries are those with compatible final spaces. It will be useful to recall the following result by R. Douglas \cite{douglas}:
\begin{rem}\label{resultado douglas}
Douglas' theorem considers the existence of solutions of the operator equation $AX=B$ \cite{douglas}: let $A,B\in\b(\h)$,  then the  following conditions are equivalent:
\begin{enumerate}
\item
there exists $X\in\b(\h)$ such that $AX=B$;
\item
$R(B)\subset R(A)$;
\item
there exists $\lambda>0$ such that $BB^*\le \lambda AA^*$.
\end{enumerate}
\end{rem}

\begin{teo}\label{proposicion 22}
Let $T\in\ii_A$. Denote by $\TT$ its (isometric) extension to $\l$. The following are equivalent:
\begin{enumerate}
\item $T$ is $A$-adjointable;
\item$\TT^*(\h)\subset\h$;
\item
$R(T)$ is a compatible subspace. 
\item
$R(T^*A)=R(A)$;
\item
there exists $\lambda>0$ such that 
$$
T^*A^2T\le \lambda A^2.
$$
\end{enumerate}
\end{teo}
\begin{proof}
The equivalence between the first two conditions is clear: if $T$ has an $A$-adjoint $S$, then $\TT^*$ and $S$ coincide on $\h$.    Then $\TT^*(\h)\subset \h$ and $S=\TT^*|_\h$.  If $\TT^*(\h)\subset \h$, then $\TT^*|_\h$ is the $A$-adjoint of $T$.  Note that therefore the final projection $\PP=\TT\TT^*$ of $\TT$ leaves $\h$ invariant: $\PP(\h)\subset\h$. Thus, $\PP$ induces an ($A$-symmetric) idempotent $P=\PP|_\h$ in $\h$. Clearly, $P=TS$, and  $R(P)=R(T)$: $P=TS$ implies that $R(P)\subset R(T)$; on the other hand, $\TT\TT^*\TT=\TT$, so that, restricting to $\h$, $TST=PT=T$ ans therefore $R(T)\subset R(P)$. Summarizing, $R(T)$ is the range of an $A$-symmetric projection, i.e., a compatible subspace.

Conversely, suppose that $R(T)$ is compatible and denote by $P$ the unique $A$-symmetric idempotent
with $R(P)=R(T)$. In particular, $R(T)$ is closed, and $T$ has a pseudo-inverse $S$ such that $TS=P$.
Indeed, $T|_{N(T)^\perp}\to R(T)$ is an isomorphism between Banach spaces,
thus it has a bounded inverse $T'$. Put $S$
$$
S=\left\{ \begin{array}{l} T' \hbox{ in } R(P)=R(T) \\ 0 \hbox{ in } N(P) \end{array} \right. .
$$
Then $TS$ equals the identity in $R(P)$ and is zero in $N(P)$, i.e., $TS=P$. Note that $PT=T$, and that $P^*A=AP$. Then $T^*AT=A$ implies that
$$
AS=T^*ATS=T^*AP=T^*P^*A=(PT)^*A=T^*A,
$$
i.e., $S$ is the $A$-adjoint of $T$.

The equivalence with the last two conditions follows using Douglas' result: $T$ is $A$-adjointable if and only if there exists a solution $X$ to the operator equation $AX=T^*A$ (i.e., $B=T^*A$), which occurs if and only if $R(T^*A)\subset R(A)$, or equivalently, there exists $\lambda>0$ such that 
$$
BB^*=T^*A^2T\le \lambda A^2.
$$
The former condition $R(T^*A)\subset R(A)$ is in turn equivalent to
$$
R(A)=R(T^*AT)\subset R(T^*A)\subset R(A),
$$
i.e., $R(A)=R(T^*A)$.
\end{proof}
\begin{rem}
Note that if $T\in\ii_A^a$, then $T$ is injective and has closed range, i.e. $T$ is bounded from below. Note the following example, of an operator in $\ii_A$, with closed range, but whose range is not a compatible subspace of $\h$.
\end{rem}
\begin{ejem}
Let $\h=H_0^1(0,1)$, the subspace of the Sobolev space $H^1(0,1)$ obtained as the closure of the smooth functions in $(0,1)$  with compact support. Let $\l=L^2(0,1)$. With the same argument as in  \cite{compatible grassmannian} (Example 3.7), it can be shown that 
$$
\s=\{f\in\h_0^1(0,1): f\equiv 0 \hbox{ in } [1/2,1)\}
$$ is a (closed) non compatible subspace of $\h\subset\l$. Let us sketch how this is proved. Let $f_0$ be a $C^\infty$ function in $(0,1)$, of compact support, which equals $1$ on an interval centered a $t=\frac12$. Then $f_0=f_0\ \chi_{(0,1/2)}+f_0\ \chi_{[1/2,1)}$ is an orthogonal sum in $\l$, with $f_0\ \chi_{(0,1/2)}\in\overline{\s}$ (the closure of $\s$ in $\l$). Thus $\PP_{\overline{\s}}(f_0)=f_0\ \chi_{(0,1/2)}$, which does no belong to $\h$, i.e., $\PP_{\overline{\s}}(\h)\not\subset \h$, and $\s$ is not compatible.
Consider the operator $\VV$ in $\b(\l)$ given by
$$
\VV f(t)=\left\{\begin{array}{ll} \sqrt{2} \ f(2t) & \hbox{ if } t\in(0,1/2) \\ 0 & \hbox{ if } t\in [1/2,1). \end{array}\right.
$$
Note that
$$
\|\VV f\|_2=\int_0^{1/2} |\VV f(t)|^2 dt\stackrel{2t=s}{=} \int_0(\sqrt{2})^2|f(s)|^2\frac12 ds=\|f\|^2,
$$
i.e. $\VV$ is an isometry of $\l$. Clearly $\VV$ preserves smooth funcions of compact support, $\VV(\h)\subset\h$. Also it is clear that $\VV(\h)$ consists of all functions in $H^1(0,1)$ with compact support contained in $(0,1/2)$,  i.e., $\VV(\h)=\s$. Thus,  $V=\VV|_\h$ belongs to $\ii_A$, has closed range, but is not adjointable.
\end{ejem}

\begin{ejems}

\noindent

\begin{enumerate}
\item
Example 2.4.1, where $T$ is the shift operator and  $\h\subset\l$ are, respectively, the Dirichlet and the Hardy space of the circle, can be generalized. 
Let $T\in\ii_A$ such that $R(T)$ is closed and has finite co-dimension. Then $T\in\ii_A^a$. Indeed, in \cite{compatible grassmannian} it was shown that closed subspaces with finite co-dimension are compatible.
\item
There are also examples of $T\in\ii_A^a$ where $R(T)$ has infinite co-dimension. For instance, consider the shift (of infinite multiplicity) $\TT$ in $\l=\ell^2$, $\TT e_n=e_{2n}$. Consider as before, $\h=\{(x_n)\in\ell^2: \sum_{n=1}^\infty n|x_n|^2<\infty\}$. Clearly $\TT(\h)\subset\h$. Note that  $(\TT^*x)_n=x_{2n}$, and then
$$
\langle \TT^*x,\TT^*x\rangle_\h=\sum_{n=1}^\infty n |x_{2n}|^2<\sum_{n=1}^\infty 2n |x_{2n}|^2\le \langle x,x\rangle_\h,
$$
i.e., $\TT^*(\h)\subset\h$. Then $T=\TT|_\h\in\ii_A^a$.
\end{enumerate}
\end{ejems}
\begin{rem}
Again, using Douglas' result, it holds that $T\in\ii_A$ is always $A^{1/2}$-adjointable:
$$
T^*(A^{1/2})^2T=T^*AT=A=(A^{1/2})^2,
$$
which is the second condition of Douglas for $\lambda=1$.
\end{rem}

Note the evident facts that if $T\in\ii_A^a$, then $T^n\in\ii_A^a$ for $n\ge 1$; also, $T^\sharp T=1$ in $\h$.

\section{Wold decomposition of $A$-adjointable isometries}

Recall the Wold decomposition of an isometry (see for instance \cite{nagyfoias}):
given an isometry $\VV$ acting in $\l$, there exists an orthogonal decomposition 
$$\l=\l_0\oplus\l_1$$
such that $\l_0,\l_1$ reduce $\VV$, $\VV|_{\l_0}$ is a unitary operator and $\VV|_{\l_1}$ is a shift
($\l_1=\bigoplus\limits_{n=0}^\infty \VV^n\l_w$,
where $\l_w=\l\ominus\VV\l$ is the so called {\it wandering space} of $\VV$).
This decomposition is unique, $\l_1$ is determined by $\VV$,
as seen in the above formula; also $\l_0=\bigcap_{n=0}^\infty \VV^n\l$.

The next result shows that the Wold decomposition of $\VV$ in $\l$ induces an analogous compatible decomposition for $V$ in $\h$.

\begin{teo}
Let $V\in\ii_A^a$ and $\VV$ its extension to an isometry of $\l$, then the Wold decomposition $\l=\l_0\oplus\l_1$ of $\VV$ induces a (direct sum, $A$-orthogonal) decomposition
$$
\h=\h_0\oplus\h_1, \ \h_0=\l_0\cap\h, \ \h_1=\h\cap\l_1,
$$
 of compatible subspaces. $V\h_0=\h_0$, $V\h_1\subset\h_1$, so that $V|_{\h_0}$ is invertible  in $\h_0$. 
\end{teo}
\begin{proof}
The wandering subspace $\l_w=\l\ominus\VV\l$ is the orthogonal complement of the range of $\VV$, therefore its intersection with $\h$ is the $A$-orthogonal complement $R(V)^{\perp_A}$ of $R(V)$ (which are both compatible subspaces of $\h$). Note that for
$n\ge 1$, 
$$
V^n(R(V)^{\perp_A})=\VV^n(R(V)^{\perp_A}) \hbox{ is dense in } \VV^n\l_w.
$$
Thus the subspaces $V^n(R(V)^{\perp_A})$ are in direct sum for different $n$, and their sum
$$
\bigoplus_{n=0}^\infty V^n(R(V)^{\perp_A}) \hbox{ is dense in } \bigoplus_{n=0}^\infty \VV^n\l_w.
$$
Note also that since $V^{n}\in\ii_A^a$,  $V^n(R(V)^{\perp_A})$ is a compatible subspace (in particular, closed). Therefore the $A$-orthogonal sum (and direct sum in $\h$) $\bigoplus_{n=0}^\infty V^n(R(V)^{\perp_A})$ is a compatible subspace.

Next, $\VV\l_0=\l_0$, and thus $V(\h_0)\subset\VV\l_0=\l_0$, together with $V(\h_0)\subset \h$,  implies $V(\h_0)\subset\h\cap\l_0=\h_0$. By the same reason, since $V^\sharp=\VV^*|_\h$. Also $V^\sharp(\h_0)\subset\h_0$. Moreover,
$$
\h_0=V^\sharp V(\h_0)\subset V^\sharp(\h_0)\subset\h_0.
$$
Since clearly $\h_0\subset R(V)$, then $VV^\sharp=P_V$ (the $A$-symmetric idempotent in $\h$ with range equal to $R(V)$) satisfies $VV^\sharp\h_0=\h_0$. Then
$$
\h_0=VV^\sharp(\h_0)\subset V(\h_0)\subset\h_0.
$$
That is, $V(\h_0)=\h_0$, i.e. $V|_{\h_0}$ is invertible.
\end{proof}
\begin{rem}
In the above theorem it was shown that $V|_{\h_0}$ is invertible in $\h_0$. The operator $A$ does not necessarily leave $\h_0$ invariant, nevertheless $A$ induces an inner product in $\h_0$. For this induced inner product (which is implemented by the compression of $A$ to $\h_0$: if $f,g\in\h_0$, $\langle f,g\rangle_A=\langle Af,g\rangle=\langle AP_{\h_0}f,P_{\h_0}g\rangle=\langle P_{\h_0}AP_{\h_0}f,g\rangle$), the restriction $V|_{\h_0}$ is a isometric and onto, i.e. $V|_{\h_0}$ belongs to the group $\u_{P_{\h_0}AP_{\h_0}}$ of the Hilbert space $\h_0$.
\end{rem}

\section{Regular structure of $\ii_A^a$}

In the introduction it was observed that $\u_A$ acts on  $\ii_A^a$: if $G\in\u_A$ and $T\in\ii_A^a$, then $GT\in\ii_A^a$, being a composition of $A$-adjointable isometries.

Let us recall in the next remark, certain facts on $A$-orthogonal projections (see \cite{compatible grassmannian})

\begin{rem}
Let 
$$
\p_A=\{P\in\b(\h):  P^2=P, P \hbox{ is } A-\hbox{symmetrizable}\}.
$$
Recall that $\b_A(\h)$ denotes the algebra of operators acting in $\h$ that are  $A$-adjointable. It is an involutive Banach algebra with the norm $|T|=\max\{\|T\|, \|T^\sharp\|\}$. Clearly $|T^\sharp|=|T|$. The set $\p_A$ is a complemented $C^\infty$-submanifold of $\b_A(\h)$, and a homogeneous space of $\u_A$ (which in turn is a Banach-Lie group and a $C^\infty$ submanifold of $\b_A(\h)$). This means that for each $P_0\in\p_A$, the map 
$\pi_{P_0}:\u_A\to\p_A$, $\pi_{P_0}(G)=GP_0G^{-1}$, is a submersion: it is surjective onto the orbit of $P_0$ by the action, which is a  union of connected components. In particular, it has smooth local cross sections. That is, there exists a radius $r_{P_0}>0$  and a map $\sigma_{P_0}$,
$$
\sigma_{P_0}:\{P\in\o_A: |P-P_0|<r_{P_0}\}\subset\p_A\to\u_A,
$$
with the following properties:
\begin{itemize}
\item
$\sigma_{P_0}$  is a $C^\infty$ (local) cross section of $\pi_{P_0}$: 
$$
\pi_{P_0}(\sigma_{P_0}(P))=\sigma_{P_0}(P) P_0  \sigma_{P_0}(P)^{-1}=P, \hbox{ for }  P\in\p_A \hbox{ with } |P-P_0|<r_{P_0};
$$
\item
the map $\sigma_{P_0}$ extends to an open ball {\it of $\b_a(\h)$} centered at $P_0$, as a $C^\infty$ map with values in the invertible group of $\b_a(\h)$;
\item
if one moves $P$ and $P_0$, locally, the element $\sigma_{P_0}(P)$ is $C^\infty$ in both variables.
\end{itemize}
\end{rem}

The first two conditions, in fact, imply that $\pi_{P_0}$ is a submersion and $\p_A$ a complemented submanifold:
\begin{lem}\label{lema auxiliar}
Let $\a$ be a Banach algebra, $\u$ be a Banach-Lie subgroup of the invertible group $\a^\times$ of $\a$, and $X$ a Banach space on which $\u$ acts smoothly. Let $x_0\in X$ be a fixed element. Suppose that the map
$$
\pi_{x_0}:\u\to \o_{x_0}:=\{u\cdot x_0: u\in\u\}
$$
has a continuous cross section $\sigma_{x_0}$ defined on an neighbourhood of $x_0$ in $\o_{x_0}$ in the relative topology induced by $X$. If $\sigma_{x_0}$ can be extended to a smooth map in an open neighbourhood of $x_0$ in $X$, then $\o_{x_0}$ is a complemented smooth submanifold of $X$ and the map $\pi_{x_0}$ is a submersion.
 \end{lem}  
\begin{proof}
The hypothesis imply that $\pi_{x_0}$ is open as a map $\u\to\o_{x_0}$, and is a smooth map that splits,  regarded as a map $\u\to X$. 
\end{proof}
Let us fix $T_0\in\ii_A^a$, and construct the extendable local cross sections of 
\begin{equation}\label{mapa pi}
\pi_{T_0}:\u_A\to \ii_A^a, \ \pi_{T_0}(G)=GT_0.
\end{equation} 
For $T\in\ii_A^a$,  denote by $P_T:=TT^\sharp$. Clearly $P_T$ is an $A$-symmetric projection, which extends to the orthogonal projection $\PP_{R(\TT)}$ onto the range of $\TT$ in $\l$. 

If $T$ is close to $T_0$, then $P_T$ is close in $P_{T_0}$. Explicitly,
$$
|P_T-P_{T_0}|=|TT^\sharp-T_0T_0^\sharp|\le |TT^\sharp-TT_0^\sharp|+|TT_0^\sharp-T_0T_0^\sharp|\le |T||T^\sharp-T_0^\sharp|+|T-T_0||T_0^\sharp|
$$
$$
=|T-T_0|(|T|+|T_0|)\le |T-T_0|(|T-T_0|+2|T_0|),
$$
In particular, there exists $\delta_{T_0}$, which depends on $T_0$, such that if $|T-T_0|<\delta_{T_0}$, then $|P_T-P_{T_0}|<r_{P_{T_0}}$ (the radius given in the above remark).

Suppose that $|T-T_0|<\delta_{T_0}$, and denote by $G_T=\sigma_{P_{T_0}}(P_T)\in\u_A$. Clearly $G_T$ is a $C^\infty$ map of $T$ which satisfies that $G_TP_{T_0}G_{T_0}^{-1}=P_T$. Then $T':=G_T^{-1} T$ is an element of $\ii_A^a$ which has final projection 
$$
T'T'^\sharp=G_T^{-1}TT^\sharp (G^{-1})^\sharp=G_T^{-1}P_TG_T=P_{T_0}.
$$

Two elements $T',T_0$ of $\ii_A^a$
with the same final projection $P_{T_0}$ are conjugate by the action of $\u_A$:
there exists $H=H_{T_0}(T')\in\u_A$ such that $HT_0=T'$. Pick, for instance
$$
H=T'T_0^\sharp +(1-P_{T_0}).
$$
Note that
$$\begin{array}{ccl}
HH^\sharp &=& T'T_0^\sharp T_0T'^\sharp+T'T_0^\sharp(1-P_{T_0})+(1-P_{T_0})T_0T'^\sharp+(1-P_{T_0}) \\
&=& T'T'^\sharp+(1-P_{T_0})=P_{T_0}+1-P_{T_0} \\
&=&1,
\end{array}$$
because $(1-P_{T_0})T'=T_0^\sharp(1-P_{T_0})=0$. Similarly $H^\sharp H=1$. Moreover,
$$
HT_0=T'T_0^\sharp T_0+(1-P_{T_0})T_0=T'.
$$
Clearly $H=H_{T_0}(T')$ is a $C^\infty$ map in terms of $T'$ and $T_0$, and thus a $C^\infty$ map in terms of $T$.

Thus, $H\in\u_A$ and $HT_0=G_T^{-1}T$, i.e.,
$T=G_THT_0$. Define
\begin{equation}\label{sigma}
\sigma_{T_0}(T)=G_TH=G_T(G_T^{-1}TT_0^\sharp+(1-P_{T_0})=TT_0^\sharp+G_T(1-P_{T_0}) , \ \hbox{ for } T\in\ii_A^a, \ |T-T_0|<\delta_{T_0}.
\end{equation}

\begin{teo}\label{estructura diferenciable}
Let $T_0\in\ii_A^a$. Then the orbit 
$$
\o_{T_0}:=\{GT_0: G\in\u_A\}
$$ is a $C^\infty$ complemented submanifold of $\b_A(\h)$, and the map
$$
\pi_{T_0}:\u_A\to \o_{T_0}, \ \pi_{T_0}(G)=GT_0
$$
is  a $C^\infty$ submersion. The orbit $\o_{T_0}$ is a  union of connected components of $\ii_A^a$, which implies that $\ii_A^a$ is a $C^\infty$ complemented submanifold of $\b_A(\h)$
\end{teo}
\begin{proof}
The fact that $\o_{T_0}$ is a submanifold follows using Lemma \ref{lema auxiliar}, noting that the cross section $\sigma_{T_0}$ defined in (\ref{sigma}) clearly extends to a $C^\infty$ map defined on a neighbourhood of $T_0$ in $\b_A(\h)$. Let us check fact that $\o_{T_0}$ is a union of connected components of $\ii_A^a$. In \cite{compatible grassmannian} it was observed that the local cross section $P_T\mapsto G_T$ for $\pi_{P_0}$ satisfies that  that $G_{T_0}=1$, and therefore $G_T$ belongs to the connected component of the identity in $\u_A$ if $T$ is close to $T_0$. This implies that if $T$ is close enough to $T_0$, then $\sigma_{T_0}(T)$ belongs to the connected component of the identity of $\u_A$. This observation implies that any $T$ close to $T_0$, can be connected to $T_0$ by means of a continuous path in the orbit.  
\end{proof}

\section{The extension method of M.G. Krein for \\ symmetrizable operators}

In this section  the extension method of M.G. Krein \cite{krein} is considered (see also Section 125 of the classic book \cite{riesznagy} for a detailed exposition). Krein shows that if $S:\l_0\subset\l\to \l$ is a contractive symmetric operator, then there exists a symmetric extension $S:\l\to\l$ which is also a contraction. Later on, other authors elaborated on this problem (see for instance \cite{davis et al} were the solutions are parametrized, or the works of Parrott  \cite{parrott}). In this section, following the ideas of M.G. Krein with slight modifications, it will be shown that an $A$-symmetric operator defined on a $A$-compatible subspace $\h_0\subset\h$, which is contractive for the norm induced by $\langle \ \ , \ \ \rangle_A$, can be extended to a contraction for this norm to the whole space $\h$. Or rather, and equivalently, we shall work in the bigger Hilbert space $\l$, extending to a symmetric contraction of $\l$ which 
leaves $\h$ invariant.

\begin{lem}\label{krein method}
Let $\XX$ be a selfadjoint operator in $\l$ such that $\XX(\h)\subset\h$, and let $\PP$ be a projection in $\l$ with range $\l_0$, such that $P=\PP|_\h$ is an idempotent in $\h$ with range $\h_0$. Suppose that $\|\XX\PP\|=1$. Then there exists a selfadjoint operator $\ZZ$ in $\l$ such that $\ZZ(\h)\subset\h$, $\ZZ\PP=\XX\PP$ and $\|\ZZ\|=1$.
\end{lem}
\begin{proof}
For $m>0$, let us denote by $\XX_m=\frac{1}{m}\XX$,  $m$ will be adjusted in the process. Denote by  $\h_0=\PP(\h)$. Let $\BB_0=\PP\XX_m\PP:\l_0\to\l_0$. Then $\BB_0$ extends to an operator in $\l$ with the same norm: $\bar{\BB}_0=\PP\XX_m:\l\to\l_0$. Note that $\bar{\BB}_0$ leaves $\h$ invariant, because $\XX_m$ and $\PP$ do.

Next  extend $\BB_1:=\PP^\perp\XX_m\PP:\l_0\to\l_0^\perp$ to the whole space $\l$ (not enlarging the norm of $\BB_1$ and leaving $\h$ invariant). Consider the inner product
$$
[\varphi, \gamma]:=\langle\varphi,\gamma\rangle_\l-\langle\bar{\BB}_0\varphi,\bar{\BB}_0\gamma\rangle_\l.
$$
Note that
$$
[\varphi, \gamma]:=\langle\varphi,\gamma\rangle_\l-\langle\PP\XX_m\varphi,\PP\XX\gamma\rangle_\l=\langle(I-\XX_m\PP\XX_m)\varphi,\gamma\rangle_\l.
$$
Take $m>0$ so that $\|\XX_m\PP\XX_m\|<1$, then $I-\XX_m\PP\XX_m$ is (positive and) invertible. Then $(\l,[\ \ , \  \ ])$ is a Hilbert space, whose norm is equivalent to the original norm induced by $\langle \ \ , \  \ \rangle_\l$. The operator $\BB_1$ is bounded by   (for $\varphi=\PP\varphi\in\l_0$)
$$
\|\BB_1\varphi\|^2=\langle \PP^\perp\XX_m\PP\varphi,\PP^\perp\XX_m\PP\varphi\rangle_\l=\langle\XX_m(1-\PP)\XX_m\PP\varphi,\PP\varphi\rangle_\l
$$
$$
=\langle\XX_m^2\PP\varphi,\PP\varphi\rangle_\l-\langle\XX_m\PP\XX_m\PP\varphi,\PP\varphi\rangle_\l.
$$
Since $\|\XX\PP\|=\frac{1}{m}$, the first term is bounded by 
$$
\langle\XX_m^2\PP\varphi,\PP\varphi\rangle_\l=\langle\PP\XX_m^2\PP\varphi,\varphi\rangle_\l\le\|\PP\XX_m^2\PP\|\|\varphi\|^2=\frac{1}{m^2}\|\varphi\|^2.
$$ 
Pick $m$ so that $m\ge 1$, one has that 
$$
\|\BB_1\varphi\|^2\le \frac{1}{m^2}\langle\varphi,\varphi\rangle_\l-\langle\XX_m\PP\XX_m\varphi,\varphi\rangle_\l \le \frac{1}{m^2}\langle\varphi,\varphi\rangle_\l-\frac{1}{m^2}\langle\XX_m\PP\XX_m\varphi,\varphi\rangle_\l=\frac{1}{m^2}[\varphi,\varphi].
$$
Therefore $\BB_1$ induces an operator $\b:(\l_0,[\ \ , \ \ ])\to (\l_0^\perp, \langle\ \ , \  \ \rangle_\l)$ with norm less than or equal to $\frac{1}{m}$. Since at the set level, this mapping coincides with $\BB_1$, $\b$ maps $\h_0$ into $\l_0^\perp\cap\h$. Let $\Pi:(\l,[\ \ , \ \ ])\to (\l_0, [\ \ , \ \ ])$ denote the $[\ \ , \ \ ]$-orthogonal projection. Since $[ \ \ ,  \ \ ]=\langle (1-\XX_m\PP\XX_m) \ \ , \ \ \rangle_\l$, the $[\ \ , \ \ ]$ adjoint of an operator $\t$ acting in $(\l,[\ \ , \ \ ])$ is given by $\t^\sharp=(1-\XX_m\t^*\XX_m)^{-1}\t^*(1-\XX_m\t^*\XX_m)$, with $\t^*$ the $\langle\ \ , \ \ \rangle_\l$- adjoint of $\t$. It is known (see for instance \cite{ando}), that the orthogonal projection onto the range of an idempotent $Q$ in a Hilbert space, is given by the formula
$$
P_{R(Q)}=Q(Q+Q^*-1)^{-1}.
$$
In our case, this implies that 
$$
\Pi=\PP\{\PP+(1-\XX_m\PP\XX_m)^{-1}\PP(1-\XX_m\PP\XX_m)-1\}^{-1}.
$$
If one further adjusts $m>0$, one has that $\Pi(\h)\subset\h$. To shorten the writing, denote by $Q_m$ the idempotent $(1-\XX_m\PP\XX_m)^{-1}\PP(1-\XX_m\PP\XX_m)$. First one needs that $Q_m(\h)\subset\h$.  Since $1-\XX_m\PP\XX_m$ leaves $\h$ invariant, it  induces an   operator in $\h$. If one further enlarges $m$ so that $X_mPX_m=\XX_m\PP\XX_m|_\h$ has norm strictly less than $1$ in $\b(\h)$, $1-X_mPX_m$ will be invertible in $\b(\h)$, i.e. $(1-\XX_m\PP\XX_m)^{-1}$ leaves $\h$ invariant. It follows that $Q_m(\h)\subset\h$.

Next, note the elementary identity
\begin{equation}\label{formula idempotentes}
(\PP+Q_m-1)^2=1-(\PP-Q_m)^2,
\end{equation}
which only uses the fact that $\PP$ and $Q_m$ are idempotents. In view of (\ref{formula idempotentes}), in order to to have that  $(\PP+Q_m-1)^2$, and therefore also $\PP+Q_m-1$, is invertible, it suffices to adjust $m$ so that the norm of $\PP-Q_m$ as an operator restricted to $\h$, is strictly less than $1$. This clearly can be done, since $X_mPX_m\to 0$ in $\b(\h)$ as $m\to +\infty$. It follows that $\Pi(\h)\subset\h$, for $m$ sufficiently large.

Denote by $J:(\l,\langle\ \ , \ \ \rangle_\l) \to (\l, [\ \ , \ \ ])$ the identity mapping. Note that $J$ is contractive:
$$
[\varphi, \varphi]=\langle(1-\XX_m\PP\XX_m)\varphi,\varphi\rangle_\l\le\langle\varphi,\varphi\rangle_\l,
$$
because $\|\XX_m\|<1$. Clearly $J(\h)\subset\h$. Therefore, if  one puts
$$
\bar{\BB}_1:=\b\Pi J:\l\to\l_0^\perp,
$$
where both spaces are considered with their original inner product $\langle\ \ , \  \ \rangle_\l$, one has  that (since $\Pi$ and $J$ are contractive), 
$$
\|\bar{\BB}_1\|\le \frac{1}{m}.
$$
Note also that all the operators involved in this product leave $\h$ invariant, thus $\bar{\BB_1}(\h)\subset \h$. Put
$$
\bar{\BB}=\bar{\BB}_0+\bar{\BB}_1.
$$
It is easily verified that for $\varphi\in\l$ with $\|\varphi\|=1$, 
$$
\|\bar{\BB}\varphi\|^2=\|\bar{\BB}_0\varphi\|^2+\|\bar{\BB}_1\varphi\|^2\le \frac{1}{m^2},
$$
since $\|\bar{\BB}_0\|, \|\bar{\BB}_1\|\le\frac{1}{m}$, $R(\bar{\BB}_0)\subset\l_0$ and $R(\bar{\BB}_1)\subset\l_0^\perp$. 
Note that $\bar{\BB}_0\PP=\PP\XX_m\PP$ and that if $\varphi_0\in\l_0=R(\PP)$,
$$
\bar{\BB}_1\varphi_0=\b\Pi J\varphi_0=\b\varphi_0=\BB_1\varphi_0,
$$
because $\varphi_0\in R(\Pi)$. Thus $\bar{\BB}_1\PP=\BB_1\PP=\PP^\perp\XX_m\PP$.
It follows that 
$$
\bar{\BB}\PP=\bar{\BB}_0\PP+\bar{\BB}_1\PP=\PP\XX_m\PP+\PP^\perp\XX_m\PP=\XX_m\PP.
$$
Clearly  $\bar{\BB}(\h)\subset\h$. The last thing to fix is that $\bar{\BB}$ is not selfadjoint.   

Note that
$$
\PP\bar{\BB}=\PP\bar{\BB}_0=\PP\XX_m,
$$ 
and then  $\bar{\BB}^*\PP=\XX_m\PP$.  
$$
\ZZ_m:=\frac12(\bar{\BB}+\bar{\BB}^*)
$$
does the feat: $\ZZ_m^*=\ZZ_m$, $\ZZ_m\PP=\XX_m\PP$, $\|\ZZ_m\|\le \|\bar{\BB}\|$. In this case one needs  further to verify that $\bar{\BB}^*(\h)\subset\h$.
Note that 
$$
\bar{\BB}^*=\XX_m\PP+J^*\Pi^*\b^*,
$$
with the adjoints taken in their respective spaces. Since $\Pi$ is an orthogonal projection, $\Pi^*=\Pi$. Note that the adjoint of $J$ is given by
$$
\langle\varphi, J^*\eta\rangle_\l=[J\varphi,\eta]=\langle(1-\XX_m\PP\XX_m)\varphi,\eta\rangle_\l=\langle\varphi,(1-\XX_m\PP\XX_m)\eta\rangle_\l,
$$
 i.e., $J^*:(\l,[\ \ , \ \ ])\to (\l, \langle\  \ , \ \ \rangle_\l)$ is $J^*=(1-\XX_m\PP\XX_m)$. 
The adjoint of 
$$
\b:(\l_0,[\ \ , \ \ ])\to (\l_0^\perp, \langle\ \ , \  \ \rangle_\l)
$$
is given by
$$
[\b^*\varphi,\eta]=\langle\varphi, \b\eta\rangle_\l=\langle\varphi, \PP^\perp\XX_m\PP\eta\rangle_\l=\langle(1-\XX_m\PP\XX_m)^{-1}(1-\XX_m\PP\XX_m)\varphi,\PP^\perp\XX_m\PP\eta\rangle_\l
$$
$$
=[\varphi,(1-\XX_m\PP\XX_m)^{-1}\PP^\perp\XX_m\PP\eta],
$$ 
i.e., $\b^*=(1-\XX_m\PP\XX_m)^{-1}\PP^\perp\XX_m\PP$, which leaves $\h$ invariant for the appropriate election of $m$.

It follows that $\ZZ_m(\h)\subset\h$. Finally,  $m$ is disposed of: 
$\ZZ=m\ZZ_m$  is selfadjoint, satisfies $\ZZ\PP=\XX\PP$, $\ZZ$ leaves $\h$ invariant, and
$$
1=\|\ZZ\PP\|\le\|\ZZ\|\le 1,
$$
i.e. $\|\ZZ\|=1$.  
\end{proof}

\section{Minimality of curves in $\ii_A^a$ with given initial conditions}

In \cite{arv}, a metric was introduced in the set of isometries of a Hilbert space. If $\ii$ denotes the set of isometries in $\l$, a natural metric in $\ii$ is given by 
$$
d(V,W)=\inf\{\ell(\gamma): \gamma \hbox{ is a smooth curve in } \ii \hbox{ joining } V \hbox{ and } W\},
$$
where 
$$
\ell(\gamma)=\int_I \|\dot{\gamma}(t)\| dt.
$$
Fix $V\in\ii$. A vector $\v$   tangent to $\ii$ at $V$ is the velocity vector $\dot{\gamma}(0)$ of a curve $\gamma\subset\ii$ with $\gamma(0)=V$. A smooth curve $\gamma$ in $\ii$ can be lifted to smooth curve of  unitary operators in $\l$: $\gamma(t)=\Gamma(t)V$ for $\Gamma$ in $\u(\l)$. It follows that $\v=i XV$, for $X^*=X\in\b(\l)$. The main result established in  \cite{arv} (Theorem 2.2) states the following:

Suppose that $\|\v\|=1$, and let $P=VV^*$ be the final projection of $V$. If $Z^*=Z$ satisfies that $\|Z\|=1$ and $ZP=\v P$, then the curve
$$
\delta(t)=e^{itZ}V
$$
has minimal length along its path for $|t|\le\pi$,  among all smooth curves in $\ii$ which join the same endpoints as $\delta$.

Define the metric in $\ii_A^a$, by considering the norm of $\b(\l)$ at every tangent space: if $T\in\ii_A^a$ and $\v\in(T\ii_A^a)_T$, then
$$
|\v|=\|\boldsymbol{\v}\|,
$$
where $\boldsymbol{\v}$ is the extension of $\v$ to $\l$.
Formally, since one has considered the regular structure of $\ii_A^a$ regarding it as a submanifold of $\b_A(\h)$, one is introducing a (constant) distribution of norms in the tangent bundle, consisting of the norm defined by the $A$-inner product. Or, if one does not care about the regular structure of $\ii_A^a$, one just regards $\ii_A^a$ as a subset of the metric space $(\ii,d)$.

A direct consequence of the result cited above and Lemma \ref{krein method} is the following:
\begin{teo}\label{minimalidad}
Let $T\in\ii^a_A$, and $\v\in(T\ii_A^a)_T$ with $|\v|=1$. Then there exists a symmetrizable operator $Z\in\b_A(\h)$ such that the curve  
$$
\delta(t)=e^{it\ZZ}\TT|_\h=e^{itZ}T\in\ii_A^a
$$
satisfies that   $\delta(0)=T$,  $\dot{\delta}(0)=\v$ and  $\delta$ has minimal length along its path for $|t|\le\pi$,  among all smooth curves in $\ii^a_A$ which join the same endpoints as $\delta$. In fact, $\delta$ is minimal also in the bigger manifold $\ii$. 
\end{teo}
\begin{proof}
The tangent vector $\v$ is of the form $\v=iXT$, for some  $X\in\b(\h)$  symmetrizable operator with $|\v|=\|\XX \TT\|=1$.  We claim that there exists  $\ZZ^*=\ZZ$ with $\|\ZZ\|=1$, $\ZZ\TT=\XX\TT$ and $\ZZ(\h)\subset \h$. Indeed, note that $\ZZ\TT=\XX\TT$ if and only if $\ZZ\PP=\XX\PP$, for $\PP$ the final projection of $\TT$: 
$$
\ZZ\TT=\XX\TT\Longrightarrow\ZZ\PP=\ZZ\TT\TT^*=\XX\TT\TT^*=\XX\PP\Longrightarrow \ZZ\PP\TT=\ZZ\TT=\XX\PP\TT=\XX\PP.
$$
Also note that  
$$
\|\XX\PP\|=\|\XX\TT\TT^*\|\le\|\XX\TT\|=\|\XX\TT\TT^*\TT\|\le\|\XX\TT\TT^*\|=\|\XX\PP\|,
$$
i.e. $\|\XX\PP\|=1$.
If $Z=\ZZ|_\h$, then $\delta(t)=e^{itZ}T$ has minimal length along its path, due to Theorem 2.2 in  \cite{arv}. Clearly $\dot{\delta}(0)=ZT=XT=\v$. 
\end{proof}

\section{Components of $\ii_A^a$ and the action of the restricted group}

The connected components of the set of usual isometries in a Hilbert space are parametrized by the co-rank of the isometries: two isometries lie in the same connected component (i.e., are conjugate by the left action of the unitary group) if and only if they have the same co-rank $0\le n\le+\infty$. This fact follows easily and is a consequence  of the connectedness of the unitary group of a Hilbert space. We do not know if the group $\u_A$ is connected (or how the properties of the operator $A$ determine the components of $\u_A$): we believe that it is an interesting problem. What does hold in general, both for usual or $A$-isometries, is that two isometries are conjugate by the group action if and only if their final projections are conjugate. The argument is essentially contained in the discussion before Theorem \ref{estructura diferenciable}.
\begin{prop}\label{conjugados}
Let $T_1, T_2\in\ii_A^a$, with final $A$-orthogonal idempotents $P_1=T_1T_1^\sharp$, $P_2=T_2T_2^\sharp$. 
\begin{enumerate}
\item
There exists $G\in\u_A$ such that $GT_1=T_2$ if and only if  there exists $H\in\u_A$ such that $HP_1H^{-1}=P_2$.
\item
$T_1$ and $T_2$ lie in the same connected component of $\ii_A^a$ if and only if $P_1$ and $P_2$ lie in the same component of $\p_A$.
\end{enumerate} 
\end{prop} 
\begin{proof}
If there exists $G\in\ii_A^a$ such that $GT_1=T_2$, then 
$$
P_2=T_2T_2^\sharp=GT_1(GT_1)^\sharp=GT_1T_1^\sharp G^\sharp=GP_1G^{-1}.
$$
Conversely, if there exists $H\in\u_A$ such that $HP_1H^{-1}=P_2$, then $T_1'=HT_1$ and $T_2$ are $A$-isometries with the same final space $P_2$. Consider the operator $K$ given by $K=T_2(T_1')^\sharp+1-P_2$. Note that
$$
KK^\sharp=(T_2(T_1')^\sharp+1-P_2)(T_1'T_2^\sharp+1-P_2)=T_2(T_1')^\sharp T_1'T_2^\sharp+1-P_2,
$$
because $T_2(T_1')^\sharp(1-P_2)=0=(1-P_2)T_1'T_2^\sharp$. Since 
$$
T_2(T_1')^\sharp T_1'T_2^\sharp=T_2T_2^\sharp=P_2,
$$
then $KK^\sharp=1$. Similarly $K^\sharp K=1$, i.e. $K\in\u_A$. Moreover
$$
KT_1'=(T_2(T_1')^\sharp+1-P_2)T_1'=T_2(T_1')^\sharp T_1=T_2.
$$
Therefore, $T_2=KT_1'=KHT_1$, with $KH\in\u_A$.

In order to prove part 2., one uses the fact that both 
$$
\pi_{T_1}:\u_A\to \{GT_1: G\in\u_A\} , \ \pi_{T_1}(G)=GT_1
$$
and
$$
\pi_{P_1}:\u_A\to \{GP_1G^{-1}: G\in\u_A\} , \ \pi_{P_1}(G)=GP_1G^{-1}
$$
C$^\infty$ submersions, and that the orbits $\{GT_1: G\in\u_A\}$ and $\{GP_1G^{-1}: G\in\u_A\}$ are unions of connected components of $\ii_A^a$ and $\p_A$, respectively. Therefore a continuous path $T(t)\in\ii_A^a$ with $T(0)=T_1$ and $T(1)=T_2$ lifts to a continuous path $G(t)\in\u_A$, i.e.,  $T(t)=G(t)T_1$. Then $P(t)=G(t)P_1G^{-1}(t)$ is a continuous path in $\p_A$ joining $P_1$ and $P_2$.

Similarly, if $P(t)$ is a continuous path in $\p_A$, there exists a continuous path $H(t)\in\u_A$ such that $P(t)=H(t)P_1H^{-1}(t)$. Then as in the proof of part 1., 
$$
K(t)=T_2(H(t)T_1)^\sharp+1-P_2\in\u_A
$$
is a continuous path, and $T(t)=K(t)H(t)T_1$ is a continuous path in $\ii_A^a$, joining $T_1$ and $T_2$. 
\end{proof}
In the remaining of this section,  the action of the restricted group
$$
\u_A^\infty:=\{G\in\u_A: G-1 \hbox{ is compact}\}
$$
on $\ii_A^a$ is considered. In \cite{grupo dos normas} it was shown that $\u_A^\infty$ is a $C^\infty$-Banach-Lie group, whose Banach-Lie group is 
$$
\mathfrak{u}_A^\infty=\{iX\in\b_A(\h): X \hbox{ is symmetrizable and compact}\}.
$$
Denote by $\k(\h)$ the space of compact operators in $\h$. In \cite{grupo dos normas} it was also proved that the usual exponential map $\exp(iX)=e^{iX}$,
$$
\exp:\mathfrak{u}_A^\infty\to \u_A^\infty
$$
is surjective. This implies, in particular, that $\u_A^\infty$ is connected. Let us consider the action of the restricted group $\u_A^\infty$ on $\ii_A^a$. The first result characterizes the orbits of the restricted action. Note that these orbits are connected. 
\begin{teo}\label{orbitas restringidas}
Let $T_0\in\ii_A^a$. Then
$$
\{GT_0: G\in\u_A^\infty\}=\{T\in\ii_A^a: T-T_0\in\k(\h)\}.
$$
\end{teo}
\begin{proof}
If $T=GT_0$ for $G\in\u_A^\infty$, then clearly $T\in\ii_A^a$ and 
$$
T-T_0=GT_0-T_0=(G-1)T_0\in\k(\h).
$$
Conversely, suppose that $T\in\ii_A^a$ satisfies that $T-T_0$ is compact. Then, by  a  theorem of I.C. Gohberg and M.I. Zambickii \cite{gohberg zambickii} (see Remark \ref{gohberg})
the extension $\TT-\TT_0$ is compact in $\l$. 
Then $\TT^*(\TT-\TT_0)=1-\TT^*\TT_0$ is compact, i.e. $\TT^*\TT_0$ is a Fredholm operator in $\l$ with zero index. Then, since $\TT$ is an isometric isomorphism between $\l$ and $R(\TT)=R(\TT\TT^*)$, and $\TT_0^*$ is an isometric isomorphism between $R(\TT_0)=R(\TT_0\TT_0^*)$ and $\l$, it follows that
$$
\TT\TT^*\TT_0\TT_0^*:R(\TT_0\TT_0^*)\to R(\TT\TT^*)
$$ 
is a Fredholm operator of zero index. In \cite{compatible grassmannian} (Theorem 6.3) it was shown this implies that $P_T=TT^\sharp$ and $P_{T_0}=T_0T_0^\sharp$ lie in the same connected component (same orbit) of the restricted Grassmannian: there exists $G\in\u_A^\infty$ such that $GP_{T_0}G^{-1}=P_T$. As in Proposition \ref{conjugados}, 
let $K=T(GT_0)^\sharp+1-P_T \in\u_A$, which satisfies $KGT_0=T$. We claim that $K\in\u_A^\infty$. Indeed,
$$
K-1=T(GT_0)^\sharp-P_T=TT_0^\sharp G^{-1} -GT_0T_0^\sharp G^{-1}=(T-GT_0)T_0^\sharp G^{-1},
$$
and since $G\in\u_A^\infty$ is of the form $G=1+C$ with $C\in\k(\h)$, 
$$
T-GT_0=T-(1+C)T_0=T-T_0-CT_0\in\k(\h).
$$
\end{proof}
With a similar argument as the one used to prove that $\ii_A^a$ has differentiable structure, it can be shown that the orbit of $T_0$ under the action on $\u_A^\infty$ is a $C^\infty$-manifold and a homogeneous space of this group.
\begin{prop}\label{variedad restringida}
Given a fixed $T_0\in\ii_A^a$, the set $\{T\in\ii_A^a: T-T_0\in\k(\h)\}$ is a complemented $C^\infty$-submanifold of the affine space $T_0+\b_A(\h)\cap\k(\h)$, and the onto map
$$\pi^\infty_{T_0}:\u_A^\infty\to \{T\in\ii_A^a: T-T_0\in\k(\h)\} , \ \pi^\infty_{T_0}(G)=GT_0
$$
is a $C^\infty$-submersion
\end{prop}
\begin{proof}
We shall use Lemma \ref{lema auxiliar}. Consider the Banach algebra $\b=\mathbb{C}1+\b_A(\h)\cap\k(\h)$, with the norm $|B|=\max\{\|B\|,\|B\|^\sharp\}$. The group $\u_A^\infty$ is a subgroup of the Banach-Lie group of invertible operators in $\b$. The element
$$
K=T(GT_0)^\sharp+1-P_T \in\u_A^\infty
$$
defines a map - in fact a cross section - on a neighbourhood of $T_0$ in $\{T\in\ii_A^a: T-T_0\in\k(\h)\}$, with values in $\u_A^\infty$. Indeed, the operator $G$ is also a $C^\infty$ map in the argument $T$ (because of the regular structure of the restricted orbit of $P_{T_0}$ under the action of $\u_A^\infty$ \cite{compatible grassmannian}). Note that  the map $T\mapsto P_T=TT^\sharp$ is $C^\infty$, it is the restriction (to the orbit of $T_0$) of the global $C^\infty$ map $B\mapsto BB^\sharp$ of the ambient algebra $\b$. Clearly, then,  the local cross section $T\mapsto KG$ extends to a $C^\infty$ map defined on a neighbourhood of $T_0$ in $\b$. 

The space on which $\u_A^\infty$ acts by left multiplication, namely $\{X\in\b_A(\h): X-T_0\in\k(\h)\}$, is not a Banach space. It is the {\it affine} Banach space $T_0+\b_A(\h)\cap\k(\h)$. Therefore the argument proceeds, with slight modifications. 
\end{proof}

{\bf Aknowledgements}

I wish to thank Mar\'\i a Eugenia Di Iorio y Lucero for many helpful comments and suggestions.

This work was supported by the grant PICT 2019 04060 (FONCyT - ANPCyT, Argentina)

\end{document}